\documentclass[reqno,12pt]{amsart}
\usepackage{amscd,amsfonts,amssymb}
\usepackage{tikz}
\usetikzlibrary{decorations}
\usetikzlibrary{decorations.pathmorphing}
\usetikzlibrary{%
arrows,%
calc,%
fit,%
patterns,%
plotmarks,%
shapes.geometric,%
shapes.misc,%
shapes.symbols,%
shapes.arrows,%
shapes.callouts,%
backgrounds,%
chains,%
topaths,%
trees,%
petri,%
mindmap,%
matrix,%
folding,%
fadings,%
through,%
positioning,%
scopes,%
decorations.fractals,%
decorations.shapes,%
decorations.text,%
decorations.pathmorphing,%
decorations.pathreplacing,%
decorations.footprints,%
decorations.markings,%
shadows}
\numberwithin{equation}{section}
\newtheorem{Theorem}{Theorem}[section]
\newtheorem{Lemma}{Lemma}[section]
\newtheorem{Corollary}{Corollary}[section]

\theoremstyle{definition}

\theoremstyle{remark}
\newtheorem{Remark}{Remark}[section]
\newtheorem{Example}{Example}[section]

\newcommand{\essinf}{\mathop{\rm ess \, inf}\limits}
\newcommand{\esssup}{\mathop{\rm ess \, sup}\limits}

\newcommand{\mes}{\mathop{\rm mes}\nolimits}

\newcommand{\dist}{\mathop{\rm dist}\nolimits}

\author{A.A. Kon'kov}
\address{Department of Differential Equations,
Faculty of Mechanics and Mathematics,
Mo\-s\-cow Lo\-mo\-no\-sov State University,
Vorobyovy Gory,
Moscow, 119992 Russia;
Center of Nonlinear Problems of Mathematical Physics,
RUDN University,
Miklukho-Maklaya str. 6,
Moscow, 117198 Russia
}
\email{konkov@mech.math.msu.su}
\author{A.E. Shishkov}
\address{
Center of Nonlinear Problems of Mathematical Physics,
RUDN University,
Miklukho-Maklaya str. 6,
Moscow, 117198 Russia
}
\email{aeshkv@yahoo.com}
\thanks{The work is supported by RUDN University, Strategic Academic Leadership Program. The work of the first author is also supported by RSF, grant 20-11-20272.  }
\title[On removable singular sets]{On removable singular sets for solutions of higher order differential inequalities}
\date{}

\begin{document}

\begin{abstract}
We obtain conditions guaranteeing that weak solutions of the differential inequality 
$$
	\sum_{|\alpha| = m}
	\partial^\alpha
	a_\alpha (x, u)
	\ge
	f (x)
	g (|u|)
	\quad
	\mbox{in } \Omega \setminus S,
$$
has a removable singular set $S \subset \Omega$, where $\Omega$ is a bounded domain and $a_\alpha$, $f$, and $g$ are some functions.
\end{abstract}

\maketitle

\section{Introduction}

We study solutions of the inequality
\begin{equation}
	\sum_{|\alpha| = m}
	\partial^\alpha
	a_\alpha (x, u)
	\ge
	f (x)
	g (|u|)
	\quad
	\mbox{in } \Omega \setminus S,
	\label{1.1}
\end{equation}
where $\Omega$ is a bounded domain in ${\mathbb R}^n$, $m, n \ge 1$ are integers, $S$ is a 
compact subset of $\Omega$, and $a_\alpha$ are Caratheodory functions such that
$$
	|a_\alpha (x, \zeta)| 
	\le 
	A |\zeta|,
	\quad
	|\alpha| = m,
$$
with some constant $A > 0$ for almost all $x \in \Omega$ 
and for all $\zeta \in {\mathbb R}$.
Also let $f : \Omega \to [0, \infty)$ be a measurable function 
and $g : [0, \infty) \to [0, \infty)$ be a twice continuously differentiable function such that
$g (0) = 0$ and, in addition, $g (\zeta) > 0$, $g' (\zeta) > 0$, and $g'' (\zeta) > 0$ for all $\zeta \in (0, \infty)$.
As is customary, by $\alpha = {(\alpha_1, \ldots, \alpha_n)}$ we mean a multi-index with
$|\alpha| = \alpha_1 + \ldots + \alpha_n$ 
and
$
	\partial^\alpha 
	= 
	{\partial^{|\alpha|} / (\partial_{x_1}^{\alpha_1} \ldots \partial_{x_n}^{\alpha_n})},
$
where $x = {(x_1, \ldots, x_n)}$.

Let us denote by $B_r^x$ an open ball of radius $r > 0$ centered at $x \in {\mathbb R}^n$.
In the case of $x = 0$, we write $B_r$ instead of $B_r^0$.
For an arbitrary set $\omega \subset {\mathbb R}^n$ and a real number $r > 0$ by $\omega_r$ we mean the set of the points $x \in {\mathbb R}^n$ 
the distance from which to $\omega$ is less than $r$.

The function $u$ is called a weak solution of~\eqref{1.1} if 
$u \in L_1 (\Omega \setminus S_r)$ 
and
$f (x) {g (|u|)} \in L_1 (\Omega \setminus S_r)$ for any $r > 0$
and, moreover,
\begin{equation}
	\int_\Omega
	\sum_{|\alpha| = m}
	(-1)^m
	a_\alpha (x, u)
	\partial^\alpha
	\varphi
	\,
	dx
	\ge
	\int_\Omega
	f (x)
	g (|u|)
	\varphi
	\,
	dx
	\label{1.2}
\end{equation}
for any non-negative function $\varphi \in C_0^\infty (\Omega \setminus S)$.
In so doing, we say that the singular set $S$ is removable if
$u \in L_1 (\Omega)$,
$f (x) {g (|u|)} \in L_1 (\Omega)$, 
and inequality~\eqref{1.2} is valid for all non-negative functions $\varphi \in C_0^\infty (\Omega)$.

The problem of removability of singularities for solutions of differential equations and inequalities is a traditional area of interest for many mathematicians~[1--11]. 
Most of the papers in this area are devoted to the case of second-order differential operators~[1--7].
For higher order differential inequalities~\eqref{1.1}, sufficient conditions for the removability of a singularity are known only in the case of the power nonlinearity $g (t) = t^\lambda$~\cite{BP, Kondratiev} or in the case of an isolated singularity~\cite{KSANS}.

Singular sets for solutions of linear higher order elliptic equations were studied in~\cite{Marcus}. 
For these equations, in order to remove a singularity, some growth conditions have to be imposed on the solutions when approaching a singular set. 

In the case of an arbitrary nonlinearity, conditions for the removability of a singular set remained unknown.
The paper presented to your attention provides an answer to this question.
Note that we do not impose any requirements on the growth of solutions near a singularity.
We also do not impose any ellipticity conditions on the coefficients $a_\alpha$ of the differential operator.
Thus, our results can be applied to a wide class of differential inequalities. 

Consider the Legendre transformation
$$
	g^* (\xi)
	=
	\left\{
		\begin{aligned}
			&
			\int_{
				g' (0)
			}^{
				\xi
			}
			(g')^{-1} (\zeta)
			\,
			d \zeta,
			&
			&
			\xi > g' (0),
			\\
			&
			0,
			&
			&
			\xi \le g' (0)
		\end{aligned}
	\right.
$$
of the function $g$, 
where $(g')^{-1}$ is the inverse function to $g'$.
From the Fenchel-Young inequality, it obviously follows that
$$
	a b \le g (a) + g^* (b)
$$
for all real numbers $a \ge 0$ and $b \ge 0$.
In the case of $g (t) = t^\lambda / \lambda$, $\lambda > 1$, this inequality takes the form
$$
	a b
	\le 
	\frac{
		1
	}{
		\lambda
	} 
		a^\lambda
	+ 
	\frac{
		\lambda - 1
	}{
		\lambda
	}
		b^{\lambda / (\lambda - 1)}
$$
for all real numbers $a \ge 0$ and $b \ge 0$.

We put
$$
	\gamma (\xi)
	=
	\frac{g^* (\xi)}{\xi}
	\quad
	\mbox{and}
	\quad
	\mu (\xi)
	=
	\inf_{t \ge \xi}
	\frac{g (t)}{t}.
$$

Let $N_r (\omega)$ be the minimal number of sets of diameter at most $r$ which can cover a bounded set $\omega$. The lower and upper fractal dimensions of $\omega$ respectively are defined by
$$
	\operatorname{\underline{dim}}_F \omega
	=
	\liminf_{r \to +0}
	\frac{
		\log N_r (\omega)
	}{
		\log 1 / r
	}
$$
and
$$
	\operatorname{\overline{dim}}_F \omega
	=
	\limsup_{r \to +0}
	\frac{
		\log N_r (\omega)
	}{
		\log 1 / r
	}.
$$
It does not present any particular problem to verify that
$$
	\operatorname{\underline{dim}}_F \omega
	=
	n
	-
	\limsup_{r \to +0}
	\frac{
		\log \mes \omega_r
	}{
		\log r
	}
$$
and
$$
	\operatorname{\overline{dim}}_F \omega
	=
	n
	-
	\liminf_{r \to +0}
	\frac{
		\log \mes \omega_r
	}{
		\log r
	}.
$$
Therefore, the last two expressions can be used as  another definition of the lower and upper fractal dimensions.

If $\operatorname{\underline{dim}}_F \omega = \operatorname{\overline{dim}}_F \omega$,
then we say that $\omega$ has the fractal dimension
$$
	\operatorname{dim}_F \omega
	=
	\lim_{r \to +0}
	\frac{
		\log N_r (\omega)
	}{
		\log 1 / r
	}.
$$
In the literature, the fractal dimensions are sometimes called Minkowski dimensions or box-counting dimensions~\cite{Falconer, Bishop}.
The upper fractal dimension is also known as Kolmogorov's entropy or Kolmogorov's capacity~\cite{Falconer}.

We say~\cite{Marcus} that $\omega$ has finite $k$-dimension if
$$
	\limsup_{r \to +0}
	\frac{
		\mes \omega_r
	}{
		r^{n - k}
	}
	<
	\infty.
$$
In so doing, $\omega$ is a $k$-dimensional set if 
$\mes \omega_r / r^{n - k}$ tends to a positive finite limit as $r \to +0$~\cite{Falconer}.
If $\omega$ is a $k$-dimensional set, then it obviously has both the fractal dimension $k$ and finite $k$-dimension.

\section{Main results}

\begin{Theorem}\label{T2.1}
Suppose that
\begin{equation}
	\int_1^\infty
	g^{- 1 / m} (\zeta)
	\zeta^{1 / m - 1}
	\,
	d\zeta
	<
	\infty
	\label{T2.1.1}
\end{equation}
and
\begin{equation}
	\int_{\Omega \setminus S}
	\gamma
	\left(
		\frac{1}{f (x)}
	\right)
	dx
	<
	\infty.
	\label{T2.1.2}
\end{equation}
Also let there exist a real number $\delta > 0$ such that
\begin{equation}
	\lim_{r \to +0}
	r^{
		n
		-
		k
		- 
		m 
		- 
		\delta
	}
	G^{-1}
	\left(
		C
		r
		\essinf_{S_r \setminus S}
		f^{1 / m}
	\right)
	=
	0
	\label{T2.1.3}
\end{equation}
for all real numbers $C > 0$, where 
\begin{equation}
	k = \operatorname{\overline{dim}}_F \partial S
	\label{T3.4.4}
\end{equation}
and $G^{-1}$ is the inverse function to
\begin{equation}
	G (t)
	=
	\int_t^\infty
	g^{- 1 / m} (\zeta)
	\zeta^{1 / m - 1}
	\,
	d\zeta.
	\label{T2.1.5}
\end{equation}
Then the singular set $S$ is removable for any weak solution of~\eqref{1.1}.
\end{Theorem}

\begin{Theorem}\label{T2.2}
Suppose that conditions~\eqref{T2.1.1} and~\eqref{T2.1.2} are valid 
and, moreover, $\partial S$ has finite $k$-dimension.
If
\begin{equation}
	\limsup_{r \to +0}
	r^{n - k - m}
	G^{-1}
	\left(
		C
		r
		\essinf_{S_r \setminus S}
		f^{1 / m}
	\right)
	<
	\infty
	\label{T2.2.1}
\end{equation}
and
\begin{equation}
	\liminf_{r \to +0}
	r^m
	\mu (C r^{m - n +k})
	\essinf_{S_r \setminus S}
	f
	>
	0
	\label{T2.2.2}
\end{equation}
for all real numbers $C > 0$,
where $G^{-1}$ is the inverse function to~\eqref{T2.1.5},
then the singular set $S$ is removable for any weak solution of~\eqref{1.1}.
\end{Theorem}

\begin{Remark}\label{R2.1}
Condition~\eqref{T2.1} generalizes the well-known Keller -- Osserman condition~\cite{Keller, Osserman}.
It also arises in the study of the blow-up phenomenon for entire solutions of higher order differential inequalities~\cite{KS2019}.
\end{Remark}

Theorems~\ref{T2.1} and~\ref{T2.2} are proved in Section~\ref{proof}.
Now we dwell on some results that follow from these theorems.
A special case of~\eqref{1.1} is inequalities of the form
\begin{equation}
	\sum_{|\alpha| = m}
	\partial^\alpha
	a_\alpha (x, u)
	\ge
	f (x)
	|u|^\lambda
	\quad
	\mbox{in } \Omega \setminus S,
	\label{2.1}
\end{equation}
where $\lambda$ is a real number.

\begin{Corollary}\label{C2.1}
Suppose that $\lambda > 1$,
\begin{equation}
	\int_{\Omega \setminus S}
	f^{- 1 / (\lambda - 1)} (x)
	\,
	dx
	<
	\infty,
	\label{C2.1.1}
\end{equation}
and
$$
	\lim_{r \to +0}
	r^{
		n - k - \lambda m / (\lambda  - 1) - \delta 
	}
	\esssup_{S_r \setminus S}
	f^{ - 1 / (\lambda - 1)}
	=
	0
$$
for some real number $\delta > 0$,
where $k$ is given by~\eqref{T3.4.4}.
Then the singular set $S$ is removable for any weak solution of~\eqref{2.1}.
\end{Corollary}

\begin{Corollary}\label{C2.2}
Suppose that $\lambda > 1$, condition~\eqref{C2.1.1} is valid and, moreover, $\partial S$ has finite $k$-dimension.
If
$$
	\limsup_{r \to +0}
	r^{
		n - k - \lambda m / (\lambda  - 1)
	}
	\esssup_{S_r \setminus S}
	f^{ - 1 / (\lambda - 1)}
	<
	\infty
$$
and
$$
	\liminf_{r \to +0}
	r^{
		n - k + \lambda (m - n + k)
	}
	\essinf_{S_r \setminus S}
	f
	>
	0,
$$
then the singular set $S$ is removable for any weak solution of~\eqref{2.1}.
\end{Corollary}

Proof of Corollaries~\ref{C2.1} and~\ref{C2.2} follows directly from Theorems~\ref{T2.1} and~\ref{T2.2}.

\begin{Example}\label{E2.1}
In~\eqref{2.1}, let
\begin{equation}
	f (x)
	=
	\dist^\sigma (x, S),
	\label{E2.1.1}
\end{equation}
where $\sigma$ is a real number.
In so doing, let $\Omega$ be a bounded plane domain and $S$ be the von Koch snowflake that can be constructed by the following procedure. At first step, we divide the sides of an equilateral triangle into three equal parts. Then, on each side, we replace the middle segment with two segments equal to it so that the protrusion faces outside the triangle.
Further, this procedure is repeated with each of the segments obtained at the previous step (see figure~\ref{fig1}).

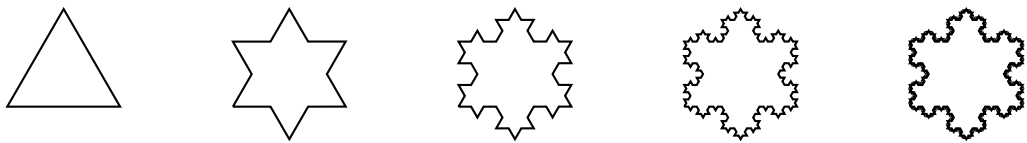
\begin{figure}[h]
\begin{tikzpicture}
[decoration=Koch snowflake,draw=black,fill=white!20,thick]
\filldraw (0,-1) -- ++(60:1.5) -- ++(-60:1.5) -- cycle ;
\filldraw decorate{(3,-1) -- ++(60:1.5) -- ++(-60:1.5) -- cycle };
\filldraw decorate{decorate{(6,-1) -- ++(60:1.5) -- ++(-60:1.5) -- cycle }};
\filldraw decorate{decorate{decorate{(9,-1) -- ++(60:1.5) -- ++(-60:1.5) -- cycle }}};
\filldraw decorate{decorate{decorate{decorate{(12,-1) -- ++(60:1.5) -- ++(-60:1.5) -- cycle }}}};
\end{tikzpicture}
\caption{Five generations of the von Koch snowflake.}
\label{fig1}
\end{figure}

It is well-known~\cite{Bishop} that
$$
	\operatorname{dim}_F \partial S = \log_3 4.
$$

By Corollary~\ref{C2.1}, if
\begin{equation}
	\lambda > 1
	\label{E2.1.2}
\end{equation}
and
$$
	\lambda
	(
		2
		-
		\log_3 4
		-
		m
	)
	-
	2
	+
	\log_3 4
	-
	\sigma
	>
	0,
$$
then the singular set $S$ is removable for any weak solution of~\eqref{2.1}.

Now, we consider the critical exponent $\lambda = 1$ in condition~\eqref{E2.1.2}. 
In~\eqref{1.1}, let the function $f$ satisfy~\eqref{E2.1.1} and
\begin{equation}
	g (t)
	=
	t
	\log^\nu (e + t).
	\label{E2.1.3}
\end{equation}
As above, we assume that $\Omega$ is a bounded plane domain and $S$ is the von Koch snowflake.
By Theorem~\ref{T2.1}, if
\begin{equation}
	\nu > m
	\label{E2.1.4}
\end{equation}
and
\begin{equation}
	\sigma < -m,
	\label{E2.1.5}
\end{equation}
then the singular set $S$ is removable for any weak solution of~\eqref{1.1}.
This result can be generalized to the case of
\begin{equation}
	g (t)
	=
	t
	\log^m (t_0 + t)
	\ldots
	\log^m
	\underbrace{\log \ldots \log}_l	(t_l + t)
	\log^\nu
	\underbrace{\log \ldots \log}_{l + 1}
	(t_{l+1} + t)
	\label{E2.1.6}
\end{equation}
where $t_0 = e$ and $t_{i + 1} = e^{t_i}$, $i = 0, 1, \ldots, l$.
In this case, in accordance with Theorem~\ref{T2.1} conditions~\eqref{E2.1.4} and~\eqref{E2.1.5} also guarantee that the singular set $S$ is removable for any weak solution of~\eqref{1.1}.
\end{Example}

\begin{Example}\label{E2.2}
In~\eqref{2.1}, let the function $f$ satisfy~\eqref{E2.1.1} and
$$
	S = \underbrace{{\mathbf C} \times \cdots \times {\mathbf C}}_{n}
$$
be the Cantor dust, 
where ${\mathbf C}$ is the usual middle thirds Cantor set constructed as follows. 
At the first step, we remove the central third $(1/3, 2/3)$ from the interval $[0, 1]$. After this procedure, the two intervals $[0, 1/3]$ and $[2/3, 1]$ remain from the interval $[0, 1]$. 
Further, at each of the subsequent steps, we remove the central third from each interval remaining at the previous step.

For $3^{-i - 1} \le r < 3^{-i}$, the set $\partial S$ can be obviously covered with $2^{n i}$ open cubes with side lengths $3^{-i + 1}$ such that the union of these cubes contains $S_r$.
As the n-dimensional Lebesgue measure of the union of these cubes does not exceed 
$2^{n i} 3^{(-i + 1)n}$, the set $\partial S$ has finite $k$-dimension, where
$$
	k
	=
	n \log_3 2.
$$

By Corollary~\ref{C2.2}, if~\eqref{E2.1.2} is valid and
$$
	\lambda (n - n \log_3 2 - m) - n + n \log_3 2 - \sigma \ge 0,
$$
then the singular set $S$ is removable for any weak solution of~\eqref{2.1}.

Now, let us examine the case of the critical exponent $\lambda = 1$ in~\eqref{E2.1.2}. 
Consider inequality~\eqref{1.1}, where $f$ and $g$ are given by~\eqref{E2.1.1} and~\eqref{E2.1.3}, respectively.
If both conditions~\eqref{E2.1.4} and~\eqref{E2.1.5} are valid, then Theorem~\ref{T2.2} implies that the singular set $S$ is removable for any weak solution of~\eqref{1.1}. In its turn, if $\sigma = -m$,
then in accordance with Theorem~\ref{T2.2} the removability of the singular set $S$ for any weak solution of~\eqref{1.1} is guaranteed by~\eqref{E2.1.4} and the inequality
$$
	n - n \log_3 2 - m \ge 0.
$$
This statement is also true for the function $g$ given by~\eqref{E2.1.6}.
\end{Example}

\begin{Example}\label{E2.3}
In~\eqref{2.1}, let the function $f$ satisfy~\eqref{E2.1.1} and 
$$
	S
	=
	\{
		x = (x_1, \ldots, x_n) 
		:
		x_1^2 + \ldots + x_k^2 \le 1,
		\:
		x_{k+1} = \ldots = x_n = 0
	\},
$$
where $0 \le k < n$ is an integer. In the partial case $k = 0$, we have $S = \{ 0 \}$.

It can be easily seen that $\partial S$ has finite $k$-dimension. Moreover, $\partial S$ is a $k$-dimensional set.
Thus, if~\eqref{E2.1.2} is valid and
\begin{equation}
	\lambda (n - k - m) - n + k - \sigma \ge 0,
	\label{E2.3.1}
\end{equation}
then in accordance with Corollary~\ref{C2.2} the singular set $S$ is removable for any weak solution of~\eqref{2.1}.

We note that, in the case of the inequality
$$
	\Delta u
	\ge
	|u|^\lambda
	\quad
	\mbox{in } \Omega \setminus \{ 0 \},
$$
where $\Omega$ is a bounded domain of dimension $n \ge 3$ containing zero,
condition~\eqref{E2.3.1} coincides with the well-known Brezis--V\'eron condition
$$
	\lambda
	\ge
	\frac{n}{n - 2}
$$
obtained in~\cite{BV}.

Now, consider inequality~\eqref{1.1}, where $f$ and $g$ are defined by~\eqref{E2.1.1} and~\eqref{E2.1.3}.
By Theorem~\ref{T2.2}, the singular set $S$ is removable for any weak solution of~\eqref{1.1} if~\eqref{E2.1.4} and~\eqref{E2.1.5} are satisfied.
In so doing, for the critical exponent $\sigma = -m$, the removability of the singular set $S$ for any weak solution of~\eqref{1.1} is guaranteed by~\eqref{E2.1.4} and the condition
$$
	n - k - m \ge 0.
$$
This result can be obviously generalized to the case of the function $g$ defined by~\eqref{E2.1.6}. \end{Example}

\begin{Example}\label{E2.4}
In~\eqref{2.1}, let $S$ be the closure of the unit ball $B_1$.
As before, we assume that the function $f$ satisfies relation~\eqref{E2.1.1}.

It is clear that $\partial S$ has finite $(n - 1)$-dimension. Thus, by Corollary~\ref{C2.2}, if~\eqref{E2.1.2} is valid and
$$
	\lambda (1 - m) - 1 - \sigma \ge 0,
$$
then the singular set $S$ is removable for any weak solution of~\eqref{2.1}.

For weak solutions of~\eqref{1.1}, where $f$ and $g$ are defined by~\eqref{E2.1.1} and~\eqref{E2.1.3}, in accordance with Theorem~\ref{T2.2} the removability of the singular set $S$ is guaranteed by conditions~\eqref{E2.1.4} and~\eqref{E2.1.5}. 
This statement is also true for $g$ given by~\eqref{E2.1.6}.
In so doing, for the critical exponent $\sigma = -m$,
the singular set $S$ is removable if $m = 1$ and $\nu > 1$.
\end{Example}

\begin{Theorem}\label{T2.3}
Suppose that conditions~\eqref{T2.1.1} and \eqref{T2.1.2} are valid. Also let there exist a real number $\delta > 0$ such that~\eqref{T2.1.3} hold for all real numbers $C > 0$, where 
\begin{equation}
	k
	=
	n
	-
	\liminf_{r \to +0}
	\frac{
		\log \mes S_r \setminus S
	}{
		\log r
	}
	\label{T2.3.1}
\end{equation}
and $G^{-1}$ is the inverse function to~\eqref{T2.1.5}.
Then the singular set $S$ is removable for any weak solution of~\eqref{1.1}.
\end{Theorem}

\begin{Theorem}\label{T2.4}
Suppose that conditions~\eqref{T2.1.1} and \eqref{T2.1.2} are valid.
Also let $k$ be a real number such that
\begin{equation}
	\limsup_{r \to +0}
	\frac{
		\mes S_r \setminus S
	}{
		r^{n - k}
	}
	<
	\infty
	\label{T2.4.1}
\end{equation}
and, moreover,~\eqref{T2.2.1} and \eqref{T2.2.2} hold for all real numbers $C > 0$,
where $G^{-1}$ is the inverse function to~\eqref{T2.1.5}.
Then the singular set $S$ is removable for any weak solution of~\eqref{1.1}.
\end{Theorem}

Theorems~\ref{T2.3} and~\ref{T2.4} are proved in Section~\ref{proof}.

\begin{Remark}\label{R2.2}
In Theorems~\ref{T2.1} -- \ref{T2.4} and Corollaries~\ref{C2.1} and~\ref{C2.2}, the condition $g (0) = 0$ can be dropped if $f (x) g (0) = 0$ almost everywhere on $S$.
\end{Remark}

\section{Proof of Theorems~\ref{T2.1} -- \ref{T2.4}}\label{proof}

In this section, we assume that $u$ is a weak solution of~\eqref{1.1}. 
Take a real number $\varepsilon \in (0, 1)$ such that the closure of $S_\varepsilon$ in ${\mathbb R}^n$ belongs to $\Omega$.
Since $S$ is a compact subset of $\Omega$, such a real number obviously exists.

We need the following simple lemma.

\begin{Lemma}\label{L3.1}
For any set $\omega \subset {\mathbb R}^n$ and real number $r > 0$ there exists a non-negative function $\psi \in C^\infty ({\mathbb R}^n)$ such that
$$
	\left.
		\psi
	\right|_{
		\omega
	}
	=
	1,
	\quad
	\left.
		\psi
	\right|_{
		{\mathbb R}^n \setminus \omega_r
	}
	=
	0,
	\quad
	\mbox{and}
	\quad
	\| \psi \|_{
		C^m ({\mathbb R}^n)
	}
	\le
	C r^{-m},
$$
where the constant $C > 0$ depends only on $n$ and $m$.
\end{Lemma}

\begin{proof}
We construct a sequence (countable or finite) of points $x_i \in \omega$, $i = 1,2,\ldots$, as follows.
Let $x_1$ be an arbitrary point belonging to $\omega$.
Assume further that $x_j$ is already constructed for all $1 \le j \le i$. If
$$
	\omega 
	\subset
	\bigcup_{j=1}^i
	B_{r / 2}^{x_j},
$$
then we stop; otherwise we take 
$$
	x_{i+1} 
	\in 
	\omega 
	\setminus
	\bigcup_{j=1}^i
	B_{r / 2}^{x_j}
$$
such that
$$
	|x_{i+1}|
	<
	\inf 
	\left
	\{ 
		|x| 
		: 
		x
		\in 
		\omega 
		\setminus
		\bigcup_{j=1}^i
		B_{r / 2}^{x_j}
	\right\}
	+
	1.
$$
Since
$
	B_{r / 4}^{x_i} 
	\cap 
	B_{r / 4}^{x_j} 
	= 
	\emptyset
$
if $i \ne j$, we obviously obtain
$$
	\omega
	\subset
	\bigcup_i
	B_{r / 2}^{x_i}.
$$
In so doing, for any point $x \in {\mathbb R}^n$ the number of the balls $B_r^{x_i}$
containing $x$ does not exceed a certain value depending only on $n$.
Really, if $x \in B_r^{x_i}$, then 
$
	B_{r / 4}^{x_i} 
	\subset 
	B_{5 r / 4}^x.
$
Therefore, the number of the balls $B_r^{x_i}$ containing $x$ can not exceed 
$
	\mes B_{5 r / 4} / \mes B_{r / 4} = 5^n.
$

Consider a non-negative function $\varphi \in C_0^\infty (B_1)$ equal to one on the ball $B_{1/2}$.
Putting
$$
	\varphi_i (x)
	=
	\varphi
	\left(
		\frac{
			x - x_i
		}{
			r
		}
	\right),
$$
we obtain
$$
	\sum_i
	\varphi_i (x)
	\ge
	1
$$
for all $x \in \omega$ and
$$
	\sum_i
	\varphi_i (x)
	=
	0
$$
for all $x \in {\mathbb R}^n \setminus \omega_r$.
Thus, to complete the proof, it remains to take
$$
	\psi (x)
	=
	\eta
	\left(
		\sum_i
		\varphi_i (x)
	\right),
$$
where $\eta \in C^\infty ({\mathbb R})$ is a non-negative function such that
$$
	\left.
		\eta
	\right|_{
		(- \infty, 0]
	}
	=
	0
	\quad
	\mbox{and}
	\quad
	\left.
		\eta
	\right|_{
		[1, \infty)
	}
	=
	1.
$$
\end{proof}

The next two lemmas generalize Lemmas~3.1 and 3.2 of paper~\cite{KSANS}, 
where the case of $\Omega = B_1$ and $S = \{ 0 \}$ was considered. 

\begin{Lemma}\label{L3.2}
Let $0 < \rho < r \le \varepsilon / 2$ be real numbers. Then
\begin{equation}
	\int_{
		S_\varepsilon \setminus S_{\varepsilon / 2}
	}
	|u|
	\,
	dx
	+
	\frac{
		1
	}{
		(r - \rho)^m
	}
	\int_{
		S_r \setminus S_\rho
	}
	|u|
	\,
	dx
	\ge
	C
	\int_{
		S_{\varepsilon / 2} \setminus S_r
	}
	f (x)
	g (|u|)
	\,
	dx,
	\label{L3.2.1}
\end{equation}
where the constant $C > 0$ depends only on $A$, $n$, $m$, and $\varepsilon$.
\end{Lemma}

\begin{proof}
We agree to denote by $C$ various positive constants which can depend 
only on $A$, $n$, $m$, and $\varepsilon$.
In accordance with Lemma~\ref{L3.1} there is a non-negative function 
$\tau \in C^\infty ({\mathbb R}^n)$ 
with the norm
$
	\| \tau \|_{
		C^m ({\mathbb R}^n)
	}
$
depending only on $n$, $m$, and $\varepsilon$ such that
$$
	\left.
		\tau
	\right|_{
		S_{\varepsilon / 2}
	}
	=
	1
	\quad
	\mbox{and}
	\quad
	\left.
		\tau
	\right|_{
		\Omega \setminus S_\varepsilon
	}
	=
	0.
$$
Analogously, there exists a non-negative function
$\psi \in C^\infty ({\mathbb R}^n)$ satisfying the conditions
\begin{equation}
	\left.
		\psi
	\right|_{
		S_\rho
	}
	=
	0,
	\quad
	\left.
		\psi
	\right|_{
		\Omega \setminus S_r
	}
	=
	1,
	\quad
	\mbox{and}
	\quad
	\| \psi \|_{
		C^m ({\mathbb R}^n)
	}
	\le
	C (r - \rho)^{-m}.
	\label{PL3.2.1}
\end{equation}
Taking
$$
	\varphi (x)
	=
	\psi (x) \tau (x)
$$
as a test function in~\eqref{1.2}, we obtain
\begin{align}
	&
	\int_\Omega
	\sum_{|\alpha| = m}
	(-1)^m
	a_\alpha (x, u)
	\psi
	\partial^\alpha
	\tau
	\,
	dx
	+
	\int_\Omega
	\sum_{
		|\alpha'| + |\alpha''| = m,
		\;
		|\alpha'| > 1
	}
	a_{\alpha' \alpha''} (x, u)
	\partial^{\alpha'}
	\psi
	\partial^{\alpha''}
	\tau
	\,
	dx
	\nonumber
	\\
	&
	\qquad
	\ge
	\int_\Omega
	f (x)
	g (|u|)
	\psi
	\tau
	\,
	dx,
	\label{PL3.2.2}
\end{align}
where $a_{\alpha' \alpha''}$ are some Caratheodory functions such that
\begin{equation}
	|a_{\alpha' \alpha''} (x, \zeta)|
	\le
	C
	|\zeta|,
	\quad
	|\alpha'| + |\alpha''| = m,
	\;
	|\alpha'| > 1,
	\label{PL3.2.3}
\end{equation}
for almost all $x \in \Omega$ and for all $\zeta \in {\mathbb R}$.
In view of~\eqref{PL3.2.1}, we have
$$
	\left|
		\int_\Omega
		\sum_{
			|\alpha'| + |\alpha''| = m,
			\;
			|\alpha'| > 1
		}
		a_{\alpha' \alpha''} (x, u)
		\partial^{\alpha'}
		\psi
		\partial^{\alpha''}
		\tau
		\,
		dx
	\right|
	\le
	\frac{
		C
	}{
		(r - \rho)^m
	}
	\int_{
		S_r \setminus S_\rho
	}
	|u|
	\,
	dx.
$$
It can also be seen that
$$
	\left|
		\int_\Omega
		\sum_{|\alpha| = m}
		(-1)^m
		a_\alpha (x, u)
		\psi
		\partial^\alpha
		\tau
		\,
		dx
	\right|
	\le
	C
	\int_{
		S_\varepsilon \setminus S_{\varepsilon / 2}
	}
	|u|
	\,
	dx
$$
and
$$
	\int_\Omega
	f (x)
	g (|u|)
	\psi
	\tau
	\,
	dx
	\ge
	\int_{
		S_{\varepsilon / 2} \setminus S_r
	}
	f (x)
	g (|u|)
	\,
	dx.
$$
Thus,~\eqref{PL3.2.2} implies~\eqref{L3.2.1}.
\end{proof}

\begin{Lemma}\label{L3.3}
Suppose that
$$
	\frac{
		1
	}{
		(r_i - r_{i + 1})^m
	}
	\int_{
		S_{r_i} \setminus S_{r_{i + 1}}
	}
	|u|
	\,
	dx
	\to
	0
	\quad
	\mbox{as } i \to \infty
$$
for some decreasing sequence of real numbers $r_i \in (0, \varepsilon / 2)$, $i = 1,2,\ldots$,
tending to zero as $i \to \infty$.
Then the singular set $S$ is removable.
\end{Lemma}

\begin{proof}
As in the proof of Lemma~\ref{L3.2}, 
by $C$ we denote various positive constants which can depend 
only on $A$, $n$, $m$, and $\varepsilon$.
From the Fenchel-Young inequality, it follows that
$$
	\int_{
		S_{\varepsilon / 2} \setminus S
	}
	|u|
	\,
	dx
	\le
	\int_{
		S_{\varepsilon / 2} \setminus S
	}
	f (x) 
	g (|u|)
	\,
	dx
	+
	\int_{
		S_{\varepsilon / 2} \setminus S
	}
	f (x) 
	g^* 
	\left(
		\frac{1}{f (x)}
	\right)
	\,
	dx.
$$
Putting $r = r_i$ and $\rho = r_{i+1}$ in Lemma~\ref{L3.2} and passing to the limit as $i \to \infty$,
we obtain
\begin{equation}
	\int_{
		S_{\varepsilon / 2} \setminus S
	}
	f (x)
	g (|u|)
	\,
	dx
	<
	\infty.
	\label{PL3.3.1}
\end{equation}
At the same time,
$$
	\int_{
		S_{\varepsilon / 2} \setminus S
	}
	f (x) 
	g^* 
	\left(
		\frac{1}{f (x)}
	\right)
	\,
	dx
	=
	\int_{
		S_{\varepsilon / 2} \setminus S
	}
	\gamma
	\left(
		\frac{1}{f (x)}
	\right)
	dx
	<
	\infty
$$
by condition~\eqref{T2.1.2}. Therefore, we have
$$
	\int_{
		S_{\varepsilon / 2} \setminus S
	}
	|u|
	\,
	dx
	<
	\infty.
$$

By Lemma~\ref{L3.1}, there are non-negative functions $\psi_i \in C^\infty ({\mathbb R}^n)$ such that
$$
	\left.
		\psi_i
	\right|_{
		S_{r_{i+1}}
	}
	=
	0,
	\quad
	\left.
		\psi_i
	\right|_{
		{\mathbb R}^n \setminus S_{r_i}
	}
	=
	1,
	\quad
	\mbox{and}
	\quad
	\| \psi_i \|_{
		C^m ({\mathbb R}^n)
	}
	\le
	C (r_i - r_{i+1})^{-m},
	\quad
	i = 1,2,\ldots.
$$

Let $\varphi \in C_0^\infty (\Omega)$ be an arbitrary non-negative function. Taking
$$
	\varphi_i (x)
	=
	\psi_i (x)
	\varphi (x)
$$
as a test function in~\eqref{1.2}, we obtain
\begin{align}
	&
	\int_\Omega
	\sum_{|\alpha| = m}
	(-1)^m
	a_\alpha (x, u)
	\psi_i
	\partial^\alpha
	\varphi
	\,
	dx
	+
	\int_\Omega
	\sum_{
		|\alpha'| + |\alpha''| = m,
		\;
		|\alpha'| > 1
	}
	a_{\alpha' \alpha''} (x, u)
	\partial^{\alpha'}
	\psi_i
	\partial^{\alpha''}
	\varphi
	\,
	dx
	\nonumber
	\\
	&
	\qquad
	\ge
	\int_\Omega
	f (x)
	g (|u|)
	\psi_i
	\varphi
	\,
	dx,
	\label{PL3.3.2}
\end{align}
where $a_{\alpha' \alpha''}$ are Caratheodory functions satisfying condition~\eqref{PL3.2.3}.

It can easily be seen that
\begin{align*}
	&
	\left|
		\int_\Omega
		\sum_{
			|\alpha'| + |\alpha''| = m,
			\;
			|\alpha'| > 1
		}
		a_{\alpha' \alpha''} (x, u)
		\partial^{\alpha'}
		\psi_i
		\partial^{\alpha''}
		\varphi
		\,
		dx
	\right|
	\\
	&
	\qquad
	\le
	\frac{
		C
		\| \varphi \|_{
			C^m (\Omega)
		}
	}{
		(r_i - r_{i + 1})^m
	}
	\int_{
		S_{r_i} \setminus S_{r_{i + 1}}
	}
	|u|
	\,
	dx
	\to
	0
	\quad
	\mbox{as } i \to \infty.
\end{align*}
At the same time, by Lebesgue's bounded convergence theorem, we obviously have
$$
	\int_\Omega
	\sum_{|\alpha| = m}
	(-1)^m
	a_\alpha (x, u)
	\psi_i
	\partial^\alpha
	\varphi
	\,
	dx
	\to
	\int_{\Omega \setminus S}
	\sum_{|\alpha| = m}
	(-1)^m
	a_\alpha (x, u)
	\partial^\alpha
	\varphi
	\,
	dx
$$
and
$$
	\int_\Omega
	f (x)
	g (|u|)
	\psi_i
	\varphi
	\,
	dx
	\to
	\int_{\Omega \setminus S}
	f (x)
	g (|u|)
	\varphi
	\,
	dx
$$
as $i \to \infty$.
Thus, ~\eqref{PL3.3.2} implies~\eqref{1.2}, where $u$ is extended by zero to the set $S$.
\end{proof}

According to Lemma~\ref{L3.3}, if the singular set $S$ is not removable, then there exist numbers $\varepsilon_0 \in (0, \infty)$ and $r_0 \in (0, \varepsilon / 2) $ such that
\begin{equation}
	\frac{
		1
	}{
		(r - \rho)^m
	}
	\int_{
		S_r \setminus S_\rho
	}
	|u|
	\,
	dx
	>
	\varepsilon_0
	\label{3.1}
\end{equation}
for all $0 < \rho < r < r_0$.

\begin{Lemma}\label{L3.4}
Let~\eqref{T2.1.1} and~\eqref{3.1} be valid. Then
\begin{equation}
	\frac{
		1
	}{
		\mes S_{\rho_2} \setminus S_{\rho_1}
	}
	\int_{
		S_{\rho_2} \setminus S_{\rho_1}
	}
	|u|
	\,
	dx
	\le
	\frac{2}{\beta}
	G^{-1}
	\left(
		C
		\beta^{1/m}
		(\rho_1 - \rho_0)
		\essinf_{
			S_{\rho_2} \setminus S
		}
		f^{1 / m}
	\right)
	\label{L3.4.1}
\end{equation}
for all real numbers $0 < \rho_0 < \rho_1 < \rho_2 < r_0$,
where 
$$
	\beta
	=
	\frac{
		\mes S_{\rho_2} \setminus S_{\rho_1}
	}{
		\mes S_{\rho_2} \setminus S_{\rho_0}
	}
$$
$G^{-1}$ is the function inverse to~\eqref{T2.1.5}, and $C > 0$ is a constant depending only on
$A$, $n$, $m$, $\varepsilon$, $\varepsilon_0$, and the first summand in the left-hand side of~\eqref{L3.2.1}.
\end{Lemma}

\begin{proof}
Consider the function
$$
	J (\rho)
	=
	\frac{
		1
	}{
		\mes S_{\rho_2} \setminus S_{\rho_1}
	}
	\int_{
		S_{\rho_2} \setminus S_\rho
	}
	|u|
	\,
	dx,
	\quad
	\rho_0 \le \rho \le \rho_1.
$$
If $J (\rho_1) = 0$, then~\eqref{L3.4.1} is obvious; 
therefore, it can be assumed that $J (\rho_1) > 0$.

We construct a finite sequence of real numbers $\{ \rho_i \}_{i=1}^l$ as follows. 
The first term $\rho_1$ of this sequence is defined in the conditions of the lemma.
Assume further that $\rho_i$ is already known for some integer $i \ge 1$. 
If $\rho_i \le (\rho_0 + \rho_1) / 2$, then we put $l = i$ and stop; otherwise we take
$$
	\rho_{i+1}
	=
	\inf 
	\{
		r \in (\rho_0 , \rho_i)
		:
		J (r)
		\le 
		2 J (r_i)
	\}.
$$
Since $J$ is a non-increasing positive function on the closed interval $[\rho_0, \rho_1]$,
this procedure must complete in a finite number of steps.

Let us denote by $C$ various positive constants which can depend 
only on $A$, $n$, $m$, $\varepsilon$, $\varepsilon_0$, and the first summand in the left-hand side of~\eqref{L3.2.1}. 
From inequalities~\eqref{L3.2.1} and~\eqref{3.1}, where $\rho = \rho_{i+1}$ and $r = \rho_i$,
it follows that
$$
	\frac{
		1
	}{
		(\rho_i - \rho_{i+1})^m
	}
	\int_{
		S_{\rho_i} \setminus S_{\rho_{i+1}}
	}
	|u|
	\,
	dx
	\ge
	C
	\int_{
		S_{\varepsilon / 2} \setminus S_{\rho_i}
	}
	f (x)
	g (|u|)
	\,
	dx
$$
for all $1 \le i \le l - 1$.
This obviously yields
$$
	\int_{
		S_{\rho_i} \setminus S_{\rho_{i+1}}
	}
	|u|
	\,
	dx
	\ge
	C (\rho_i - \rho_{i+1})^m
	\essinf_{
		S_{\rho_2} \setminus S
	}
	f
	\int_{
		S_{\rho_2} \setminus S_{\rho_i}
	}
	g (|u|)
	\,
	dx
$$
or, in other words,
\begin{equation}
	J (\rho_{i+1}) - J (\rho_i)
	\ge
	C (\rho_i - \rho_{i+1})^m
	\essinf_{
		S_{\rho_2} \setminus S
	}
	f
	\frac{
		1
	}{
		\mes S_{\rho_2} \setminus S_{\rho_1}
	}
	\int_{
		S_{\rho_2} \setminus S_{\rho_i}
	}
	g (|u|)
	\,
	dx
	\label{PL3.4.1}
\end{equation}
for all $1 \le i \le l - 1$.
Since $g$ is a convex function, we have
$$
	\frac{
		1
	}{
		\mes S_{\rho_2} \setminus S_{\rho_i}
	}
	\int_{
		S_{\rho_2} \setminus S_{\rho_i}
	}
	g (|u|)
	\,
	dx
	\ge
	g 
	\left(
		\frac{
			1
		}{
			\mes S_{\rho_2} \setminus S_{\rho_i}
		}
		\int_{
			S_{\rho_2} \setminus S_{\rho_i}
		}
		|u|
		\,
		dx
	\right)
$$
for all $1 \le i \le l - 1$.
Consequently,~\eqref{PL3.4.1} implies the inequalities
$$
	J (\rho_{i+1}) - J (\rho_i)
	\ge
	C (\rho_i - \rho_{i+1})^m
	g (\beta J (\rho_i))
	\essinf_{
		S_{\rho_2} \setminus S
	}
	f,
	\quad
	i = 1, 2, \ldots, l - 1,
$$
whence it follows that
\begin{equation}
	\frac{
		J^{1 / m} (\rho_i)
	}{
		g^{1 / m} (\beta J (\rho_i))
	}
	\ge
	C
	(\rho_i - \rho_{i+1})
	\essinf_{
		S_{\rho_2} \setminus S
	}
	f^{1 / m},
	\quad
	i = 1, 2, \ldots, l - 1.
	\label{PL3.4.2}
\end{equation}

At first, let $\rho_l > \rho_0$. In this case, we can obviously assert that $J (\rho_{i+1}) = 2 J (\rho_i)$ for all $1 \le i \le  l - 1$; therefore,
$$
	\int_{
		J (\rho_i)
	}^{
		J (\rho_{i+1})
	}
	g^{- 1 / m} (\beta \zeta / 2)
	\zeta^{1 / m - 1}
	\,
	d\zeta
	\ge
	\frac{
		2^{1 / m - 1}
		J^{1 / m} (\rho_i)
	}{
		g^{1 / m} (\beta J (\rho_i))
	},
	\quad
	i = 1, 2, \ldots, l - 1.
$$
Combining this with~\eqref{PL3.4.2}, we obtain
$$
	\int_{
		J (\rho_i)
	}^{
		J (\rho_{i+1})
	}
	g^{- 1 / m} (\beta \zeta / 2)
	\zeta^{1 / m - 1}
	\,
	d\zeta
	\ge
	C
	(\rho_i - \rho_{i+1})
	\essinf_{
		S_{\rho_2} \setminus S
	}
	f^{1 / m},
	\quad
	i = 1, 2, \ldots, l - 1.
$$
Summing the last expression over all $1 \le i \le  l - 1$, one can conclude that
\begin{equation}
	\int_{
		J (\rho_1)
	}^\infty
	g^{- 1 / m} (\beta \zeta / 2)
	\zeta^{1 / m - 1}
	\,
	d\zeta
	\ge
	C
	(\rho_1 - \rho_0)
	\essinf_{
		S_{\rho_2} \setminus S
	}
	f^{1 / m}.
	\label{PL3.4.3}
\end{equation}

Now, let $\rho_l = \rho_0$. In this case, we have $\rho_{l-1} - \rho_l \ge (\rho_1 - \rho_0) / 2$; therefore,~\eqref{PL3.4.2} implies the estimate
$$
	\frac{
		J^{1 / m} (\rho_{l-1})
	}{
		g^{1 / m} (\beta J (\rho_{l-1}))
	}
	\ge
	C
	(\rho_1 - \rho_0)
	\essinf_{
		S_{\rho_2} \setminus S
	}
	f^{1 / m}.
$$
Combining this with the inequality
$$
	\int_{
		J (\rho_{l-1})
	}^{
		2 J (\rho_{l-1})
	}
	g^{- 1 / m} (\beta \zeta / 2)
	\zeta^{1 / m - 1}
	\,
	d\zeta
	\ge
	\frac{
		2^{1 / m - 1}
		J^{1 / m} (\rho_{l-1})
	}{
		g^{1 / m} (\beta J (\rho_{l-1}))
	},
$$
we again arrive at~\eqref{PL3.4.3}. 
In its turn, by the change of variable $t = \beta \zeta / 2$, formula~\eqref{PL3.4.3} can be transformed into 
$$
	\left(
		\frac{2}{\beta}
	\right)^{1/m}
	\int_{
		\beta J (\rho_1) / 2
	}^\infty
	g^{- 1 / m} (t)
	t^{1 / m - 1}
	\,
	dt
	\ge
	C
	(\rho_1 - \rho_0)
	\essinf_{
		S_{\rho_2} \setminus S
	}
	f^{1 / m},
$$
whence~\eqref{L3.4.1} follows at once.
\end{proof}

\begin{proof}[Proof of Theorem~$\ref{T2.3}$]
Assume the converse.
Let the singular set $S$ is not removable. In this case, in accordance with Lemma~\ref{L3.3} condition~\eqref{3.1} is valid.
As in the proof of Lemma~\ref{L3.4}, we denote by $C$ various positive constants which can depend 
only on $A$, $n$, $m$, $\varepsilon$, $\varepsilon_0$, and the first summand in the left-hand side of~\eqref{L3.2.1}. 

Consider a subsequence of the sequence $r_i = 2^{-i} r_0$, $i = 1,2,\ldots$, such that
\begin{equation}
	\mes S_{r_{i_j}} \setminus S_{r_{i_j + 1}}
	\ge
	\mes S_{r_{i_j + 1}} \setminus S_{r_{i_j + 2}}
	\label{PT2.3.1}
\end{equation}
for all $j = 1,2,\ldots$.
Such a subsequence obviously exists; otherwise we have
$$
	\mes S_{r_{i + 1}} \setminus S_{r_{i + 2}}
	>
	\mes S_{r_i} \setminus S_{r_{i + 1}}
$$
for all $i \ge l$, where $l$ is a positive integer.
Hence,
$$
	\mes S_{r_{i + 1}} \setminus S_{r_{i + 2}}
	>
	\mes S_{r_l} \setminus S_{r_{l + 1}}
	>
	0
$$
for all $i \ge l$. This contradicts the fact that
$$
	\sum_{i=l}^\infty
	\mes S_{r_i} \setminus S_{r_{i + 1}}
	=
	\mes S_{r_l} \setminus S
	<
	\infty.
$$

Taking into account~\eqref{T2.3.1}, we obtain
$$
	n - k - \delta_j
	\le
	\frac{
		\log 
		\mes 
		S_{
			r_{i_j}
		} 
		\setminus 
		S
	}{
		\log r_{i_j}
	},
	\quad
	j = 1,2,\ldots,
$$
where $\delta_j \to 0$ as $j \to \infty$,
whence it follows that
\begin{equation}
	\mes S_{r_{i_j}} \setminus S
	\le
	r_{i_j}^{n - k - \delta_j},
	\quad
	j = 1,2,\ldots.
	\label{PT2.3.2}
\end{equation}
Lemma~\ref{L3.4} with 
$\rho_0 = r_{i_j + 2}$, $\rho_1 = r_{i_j + 1}$, and $\rho_2 = r_{i_j}$,
yields
\begin{equation}
	\frac{
		1
	}{
		\mes S_{r_{i_j}} \setminus S_{r_{i_j + 1}}
	}
	\int_{
		S_{r_{i_j}} \setminus S_{r_{i_j + 1}}
	}
	|u|
	\,
	dx
	\le
	\frac{2}{\beta_j}
	G^{-1}
	\left(
		C
		\beta_j^{1/m}
		r_{i_j}
		\essinf_{
			S_{r_{i_j}} \setminus S
		}
		f^{1 / m}
	\right),
	\label{PT2.3.3}
\end{equation}
where
$$
	\beta_j
	=
	\frac{
		\mes S_{r_{i_j}} \setminus S_{r_{i_j + 1}}
	}{
		\mes S_{r_{i_j}} \setminus S_{r_{i_j + 2}}
	},
	\quad
	j = 1,2,\ldots.
$$
In view of~\eqref{PT2.3.2} and the inequality
\begin{equation}
	\frac{1}{2}
	\le
	\beta_j
	\le
	1
	\label{PT2.3.4}
\end{equation}
which follows from~\eqref{PT2.3.1}, formula~\eqref{PT2.3.3} leads to the estimate
$$
	\frac{
		1
	}{
		\left( 
			r_{i_j} - r_{i_j + 1}
		\right)^m
	}
	\int_{
		S_{r_{i_j}} \setminus S_{r_{i_j + 1}}
	}
	|u|
	\,
	dx
	\le
	2^{m + 2}
	r_{i_j}^{n - k - m - \delta_j}
	G^{-1}
	\left(
		C
		r_{i_j}
		\essinf_{
			S_{r_{i_j}} \setminus S
		}
		f^{1 / m}
	\right)
$$
for all $j = 1,2,\ldots$.
According to~\eqref{T2.1.3}, the right-hand side of the last expression tends to zero as $i \to \infty$. This contradicts~\eqref{3.1}.
\end{proof}

\begin{proof}[Proof of Theorem~$\ref{T2.1}$]
We have
\begin{equation}
	S_r \setminus S
	\subset
	(\partial S)_r
	\label{PT2.1.1}
\end{equation}
for all real numbers $r > 0$, where $(\partial S)_r$ is the set of the points $x \in {\mathbb R}^n$ 
the distance from which to $\partial S$ is less than $r$.
Consequently,
$$
	\operatorname{\overline{dim}}_F \partial S
	=
	n
	-
	\liminf_{r \to +0}
	\frac{
		\log \mes (\partial S)_r
	}{
		\log r
	}
	\ge
	n
	-
	\liminf_{r \to +0}
	\frac{
		\log \mes S_r \setminus S
	}{
		\log r
	}.
$$
Thus, to complete the proof, it remains to use Theorem~\ref{T2.3}.
\end{proof}

\begin{proof}[Proof of Theorem~$\ref{T2.4}$]
Assume by contradiction that the singular set $S$ is not removable. 
Then in accordance with Lemma~\ref{L3.3} condition~\eqref{3.1} is valid.
We agree to denote by $C$ various positive constants which can depend 
only on $A$, $n$, $m$, $k$, $\varepsilon$, $\varepsilon_0$, 
the limit in~\eqref{T2.4.1},
and the first summand in the left-hand side of~\eqref{L3.2.1}. 

Let us put $r_i = 2^{-i} r_0$, $i = 1,2,\ldots$. 
By~\eqref{T2.4.1}, we have
\begin{equation}
	\mes S_{r_i} \setminus S
	\le
	C
	r_i^{n - k}
	\label{PT2.4.1}
\end{equation}
for all sufficiently large $i$.
As in the proof of Theorem~\ref{T2.3}, consider a subsequence of the sequence $\{ r_i \}_{i=1}^\infty$ satisfying condition~\eqref{PT2.3.1}.
Taking $\rho_0 = r_{i_j + 2}$, $\rho_1 = r_{i_j + 1}$, and $\rho_2 = r_{i_j}$ in Lemma~\ref{L3.4},
we obviously obtain~\eqref{PT2.3.3},
whence in accordance with~\eqref{PT2.3.4} and~\eqref{PT2.4.1} it follows that
$$
	\frac{
		1
	}{
		\left( 
			r_{i_j} - r_{i_j + 1}
		\right)^m
	}
	\int_{
		S_{r_{i_j}} \setminus S_{r_{i_j + 1}}
	}
	|u|
	\,
	dx
	\le
	2^{m + 2}
	r_{i_j}^{n - k - m}
	G^{-1}
	\left(
		C
		r_{i_j}
		\essinf_{
			S_{r_{i_j}} \setminus S
		}
		f^{1 / m}
	\right)
$$
for all $j = 1,2,\ldots$.
In view of~\eqref{T2.2.1}, this yields
$$
	\limsup_{j \to \infty}
	\frac{
		1
	}{
		\left( 
			r_{i_j} - r_{i_j + 1}
		\right)^m
	}
	\int_{
		S_{r_{i_j}} \setminus S_{r_{i_j + 1}}
	}
	|u|
	\,
	dx
	<
	\infty.	
$$
Hence, putting $r = r_{i_j}$ and $\rho = r_{i_j + 1}$ in Lemma~\ref{L3.2} and passing to the limit as $j \to \infty$, we obtain~\eqref{PL3.3.1}, whence it follows that
\begin{equation}
	\lim_{i \to \infty}
	\int_{
		S_{r_i} \setminus S_{r_{i + 1}}
	}
	f (x)
	g (|u|)
	\,
	dx
	=
	0.
	\label{PT2.4.2}
\end{equation}
Since $g$ is a convex function, we have
$$
	\frac{
		1
	}{
		\mes S_{r_i} \setminus S_{r_{i + 1}}
	}
	\int_{
		S_{r_i} \setminus S_{r_{i + 1}}
	}
	g (|u|)
	\,
	dx
	\ge
	g 
	\left(
		\frac{
			1
		}{
			\mes S_{r_i} \setminus S_{r_{i + 1}}
		}
		\int_{
			S_{r_i} \setminus S_{r_{i + 1}}
		}
		|u|
		\,
		dx
	\right)
$$
for all $i = 1,2,\ldots$.
At the same time, by~\eqref{3.1}, Lemma~\ref{L3.2} implies the estimate
$$
	\int_{
		S_{r_i} \setminus S_{r_{i + 1}}
	}
	|u|
	\,
	dx
	\ge
	C r_i^m
	\int_{
		S_{\varepsilon / 2} \setminus S_{r_i}
	}
	f (x)
	g (|u|)
	\,
	dx,
	\quad
	i = 1,2,\ldots.
$$
Hence,
$$
	\frac{
		1
	}{
		\mes S_{r_i} \setminus S_{r_{i + 1}}
	}
	\int_{
		S_{r_i} \setminus S_{r_{i + 1}}
	}
	g (|u|)
	\,
	dx
	\ge
	g 
	\left(
		\frac{
			C r_i^m
		}{
			\mes S_{r_i} \setminus S_{r_{i + 1}}
		}
		\int_{
			S_{\varepsilon / 2} \setminus S_{r_i}
		}
		f (x)
		g (|u|)
		\,
		dx
	\right)
$$
for all $i = 1,2,\ldots$.
Combining this with the inequalities
$$
	\int_{
		S_{r_i} \setminus S_{r_{i + 1}}
	}
	f (x)
	g (|u|)
	\,
	dx
	\ge
	\essinf_{
		S_{r_i} \setminus S
	}
	f
	\int_{
		S_{r_i} \setminus S_{r_{i + 1}}
	}
	g (|u|)
	\,
	dx,
	\quad
	i = 1,2,\ldots,
$$
we obtain
\begin{align*}
	\int_{
		S_{r_i} \setminus S_{r_{i + 1}}
	}
	f (x)
	g (|u|)
	\,
	dx
	\ge
	{}
	&
	\essinf_{
		S_{r_i} \setminus S
	}
	f
	\,
	\mes S_{r_i} \setminus S_{r_{i + 1}}
	\\
	&
	{}
	\times
	g 
	\left(
		\frac{
			C r_i^m
		}{
			\mes S_{r_i} \setminus S_{r_{i + 1}}
		}
		\int_{
			S_{\varepsilon / 2} \setminus S_{r_i}
		}
		f (x)
		g (|u|)
		\,
		dx
	\right)
\end{align*}
for all $i = 1,2,\ldots$.
This can obviously be written in the form
\begin{equation}
	\int_{
		S_{r_i} \setminus S_{r_{i + 1}}
	}
	f (x)
	g (|u|)
	\,
	dx
	\ge
	C
	\essinf_{
		S_{r_i} \setminus S
	}
	f
	\,
	r_i^m
	\,
	\frac{g (t_i)}{t_i}
	\int_{
		S_{\varepsilon / 2} \setminus S_{r_i}
	}
	f (x)
	g (|u|)
	\,
	dx,
	\label{PT2.4.4}
\end{equation}
where
$$
	t_i
	=
	\frac{
		C r_i^m
	}{
		\mes S_{r_i} \setminus S_{r_{i + 1}}
	}
	\int_{
		S_{\varepsilon / 2} \setminus S_{r_i}
	}
	f (x)
	g (|u|)
	\,
	dx,
	\quad
	i = 1,2,\ldots.
$$

By~\eqref{T2.4.1}, we have
$$
	t_i
	\ge
	C r_i^{m - n + k}
	\int_{
		S_{\varepsilon / 2} \setminus S_{r_i}
	}
	f (x)
	g (|u|)
	\,
	dx
$$
for all sufficiently large $i$.
In view of~\eqref{T2.2.2}, $f$ is a positive function in a neighborhood of the set $S$; therefore,
$$
	\lim_{i \to \infty}
	\int_{
		S_{\varepsilon / 2} \setminus S_{r_i}
	}
	f (x)
	g (|u|)
	\,
	dx
	>
	0.
$$
Really, if the last inequality is not valid, then the function $u$ is identically equal to zero in  $S_{\varepsilon / 2} \setminus S$. This, in turn, implies that the singular set $S$ is removable.
Thus, in accordance with~\eqref{T2.2.2} for all sufficiently large $i$ the right-hand side of~\eqref{PT2.4.4} is bounded from below by a positive constant independent of $i$.
We arrive at a contradiction with~\eqref{PT2.4.2}.
\end{proof}

\begin{proof}[Proof of Theorem~$\ref{T2.2}$]
According to~\eqref{PT2.1.1}, if $\partial S$ has finite $k$-di\-men\-si\-on, then inequality~\eqref{T2.4.1} is valid. Thus, Theorem~\ref{T2.2} follows immediately from Theorem~\ref{T2.4}.
\end{proof}

\end{document}